\documentclass[11pt]{amsart}
\usepackage{amsmath}
\usepackage{enumerate}
\usepackage{amssymb}
\usepackage{amscd}
\usepackage{mathrsfs}
\usepackage{tikz-cd}
\usepackage{color}
\usepackage[all,cmtip]{xy}
\usepackage{pstricks}

\def\bA{\mathbb A}
\def\bC{\mathbb C} 
\def\bP{\mathbb P}
 
\def\sA{\mathscr A}
\def\sF{\mathscr F}
\def\sG{\mathscr G}
\def\sI{\mathscr I}
\def\sJ{\mathscr J}
\def\sK{\mathscr K}
\def\sL{\mathscr L}
\def\sN{\mathscr N}
\def\sO{\mathscr O}
\def\sP{\mathscr P}
\def\sR{\mathscr R}
\def\sS{\mathscr S}
\def\sT{\mathscr T}
\def\sW{\mathscr W}

\def\cY{{\mathcal Y}}
\def\cD{{\mathcal D}}

\def\bk{\mathrm{k}}

\def\coker{\mathrm{coker\,}}
\def\image{\mathrm{im\,}}
\def\rk{\mathrm{rk}}
\def\length{\mathrm{l}}
\def\bzero{\mathbf{0}}
\def\lra{\longrightarrow}

\def\gm{\mathrm{G_m}}
\def\gmn{\mathrm{G_m^n}}

\def\spec{\mathrm{Spec}\,}
\def\PGL{\mathrm{PGL}}
\def\Hom{\mathrm{Hom}} 

\def\sTor{\mathscr{T}\! or}
\def\sExt{\mathscr{E}\!xt}
\def\tor{\mathrm{tor}}
\def\tf{\mathrm{t.f.}}
\def\dual{^{\scriptscriptstyle \vee}}

\def\M{\mathrm{M}}
\def\fP{{\mathfrak P}}  
\def\Quot{\mathrm {Quot}}
\def\Hilb{\mathrm{Hilb}}
 
\newtheorem{thm}{Theorem}[section]
\newtheorem{defn}[thm]{Definition}
\newtheorem{lemm}[thm]{Lemma}
\newtheorem{prop}[thm]{Proposition}

\newtheorem{rmk}[thm]{Remark}

\title[Complete stable pairs]{Moduli space of complete stable pairs}
\author{Baosen Wu}
\email{bswu@tsinghua.edu.cn}
\address{Department of Mathematical Sciences\\ Tsinghua University\\ Beijing 100084\\ P.R. China}

\begin{document}

\pagestyle{plain}
\baselineskip=15pt

\begin{abstract}
We define complete stable pairs on a smooth projective variety, and construct their moduli space. These moduli spaces have natural morphisms to the moduli of stable pairs and Quot-schemes. As an example, we show that the moduli of complete stable pairs on $\bP^1$ is an iterated blowing-up of the moduli of stable pairs, similar to the construction of the space of complete collineations. 
\end{abstract}


\maketitle


\section*{Introduction}

Let $X$ be a smooth complex projective variety. Let $V$ be a finite dimensional complex vector space. There is a moduli space parameterizing stable pairs $(\sF,\psi)$ where $\sF$ is a coherent sheaf on $X$ and 
\[  \psi: V\otimes \sO_{X} \lra \sF
\]
is a homomorphism satisfying certain stability condition. One of the important examples is stable pair of limit stability in which $\sF$ is a pure sheaf of dimension one and $\coker\psi$ is a sheaf with finite support. When $X$ is a Calabi-Yau $3$-fold, R. Pandharipande and R. P. Thomas \cite{PT} constructed stable pair invariants from such moduli spaces, they are fundamental in the study of DT-invariants of Calabi-Yau $3$-folds \cite{PT2}\cite{T09}.

In this paper we shall study complete stable pairs. A complete stable pair on $X$ is a sequence of nonzero homomorphisms 
\begin{align*} 
&\psi_0:  V\otimes \sO_X\to \sF, \\  
&\psi_1:  \ker \psi_0  \to\coker\psi_0, \\  
&\hskip 50pt \vdots  \\  
&\psi_m:  \ker \psi_{m-1}  \to\coker\psi_{m-1}  
\end{align*}
so that $(\sF, \psi_0)$ is a stable pair, and $\psi_m$ is the first surjection among $\psi_0, \cdots,\psi_m$.   

The idea of complete stable pairs comes from complete collineations or complete conics. The space of complete collineations was introduced in the early development of enumerative geometry.  It was used to solve many intricate classical geometric counting problems. In the late 20th century, the space of complete collineations was reformulated and generalized using modern algebraic geometric language. For the details see \cite{CP}\cite{Kau}\cite{Lak}\cite{MTh} and references therein. 

To study the moduli problem for complete stable pairs, one important step is to define flat families of such objects. In order to solve this issue, we use the stack of expanded pairs developed by Jun Li  \cite{Li01} in his proof of degeneration formula of GW-invariants. In this paper, the expanded pair we will use is a chain of $\bP^1$'s 
\[  \bP^1_{[m]}:=\bP^1\cup \Delta_1\cup \cdots \cup\Delta_m
\] 
where each $\Delta_i$ is a $\bP^1$ and they are glued transversely in the way that the point $\infty$ on $\Delta_i$ is identified with the point $0$ on $\Delta_{i+1}$; moreover, we fix a linear $\gm$-action on $\bP^1_{[m]}$ which play an essential role in the construction. For any complete stable pair $(\sF,\psi_0,\cdots,\psi_m)$ on $X$, we construct a $\gm$-equivariant sheaf $\widetilde\sF$ on $X\times \bP^1_{[m]}$ which is flat over $\bP^1_{[m]}$, and a $\gm$-equivariant homomorphism 
\[ \widetilde\varphi: V\otimes \sO_{X\times \bP^1_{[m]}}\to \widetilde\sF.
\]
We call such pair $(\widetilde\sF, \widetilde\varphi)$ on $X\times \bP^1_{[m]}$ an expanded stable pair.  Then we can define flat families of expanded stable pairs and formulate the moduli problem. The main result of this paper is 

\begin{thm} There is a proper algebraic space $\fP_{X}^h(V)$ parameterizing equivalence classes of complete stable pairs on $X$ with Hilbert polynomial $h$. 
\end{thm}

Let $\M_X^h(V)$ be the moduli space of stable pairs $\psi: V\otimes \sO_{X} \to \sF$ on $X$ with Hilbert polynomial $h$. Let $\Quot_X^h(V)$ be the Quot-scheme parameterizing quotients $\rho: V\otimes\sO_X\to \sF$ on $X$ with Hilbert polynomial $h$.  Then there are natural morphisms 
\[ \fP_{X}^h(V)\to \M_X^h(V), \quad  \fP_{X}^h(V)\to \Quot_X^h(V).
\] 
In some special cases, e.g., $X=\bP^1$, these morphisms are iterated blowing-up along smooth centers. It is similar to the blowing-up construction of the space of complete collineations \cite{V}. The precise relation can be used to compute invariants on Quot-schemes which will appear in subsequent work. In general, these moduli spaces are very singular, however, $\fP_{X}^h(V)$ is a natural moduli space linking the moduli of stable pairs $\M_X^h(V)$ and the Quot-scheme $\Quot_X^h(V)$. When $X$ is a Calabi-Yau $3$-fold and $\dim V=1$, we hope this provides a further understanding of the DT/PT correspondence \cite{Br}\cite{T10}. It is also one of the motivations of the current work. The ideas and methods of this paper can be generalized to other moduli problems as well. 

This paper is organized as follows. In Section 1, we give two equivalent definitions of complete stable pairs. Then we use elementary modifications and glueing of sheaves to construct expanded stable pairs.  In Section 2, we define families of expanded stable pairs, and study their moduli stack. The main result is proved in this section. In Section 3, we work out the moduli of complete stable pairs on $\bP^1$ in details.  Finally, we establish a correspondence between complete stable pairs and complete quotients on $\bP^1$ in Section 4.
 
The author thanks professor Yi Hu for discussions on complete quotients \cite{Hu}, which inspired the current work.

In this paper we work over the field of complex numbers $\bk=\bC$.


\section{Complete stable pairs}

\subsection{Complete stable pairs} 

Let $X$ be a smooth projective variety over $\bk$. Fix a finite dimensional $\bk$-vector space $V$. Recall that \cite{PT} a stable pair $(\sF,\psi)$ on $X$ consists of a coherent sheaf $\sF$ on $X$ which is pure of dimension one and a homomorphism $\psi: V\otimes \sO_{X} \to \sF$ with finite cokernel.   

\begin{defn}\label{def_csp1}
Let $(\sF, \psi)$ be a stable pair on $X$. A complete stable pair over $(\sF, \psi)$ with length $m$ consists of a sequence of nonzero homomorphisms
\begin{align*}
&\psi_0:  V\otimes\sO_X\to \sF, \\
&\psi_1:  \ker \psi_0  \to\coker\psi_0, \\
&\hskip 50pt \vdots  \\
&\psi_m:  \ker \psi_{m-1}  \to\coker\psi_{m-1},
\end{align*}
such that $\psi_0=\psi$, and $\psi_m$ is the first surjection among $\psi_0, \cdots,\psi_m$, that is, $\coker \psi_i=0$ if and only if $i=m$. 
\end{defn}

We denote a complete stable pair by $(\sF,\psi_0,\cdots,\psi_m)$. It follows by definition that lengths of complete stable pairs over $(\sF, \psi)$ are bounded by the length of $\coker\psi$.

Let $(\sF, \psi_0,\cdots,\psi_m)$ and $(\sF',\psi'_0,\cdots,\psi'_{m'})$ be two complete stable pairs. Suppose $(\sF, \psi_0)$ and $(\sF', \psi'_0)$ are equivalent as stable pairs, i.e., there is an isomorphism $u_0:\sF\cong \sF'$ so that $\psi'_0=u_0\psi_0$, then $\ker\psi_0=\ker\psi'_0$ and it induces an isomorphism $u_1: \coker\psi_0\cong\coker\psi'_0$. Assume further that the following diagram commutes for some nonzero scalar $\lambda_1\in \bk$
\[ \xymatrix{
 \ker\psi_{0}\ar[r]^{\psi_1} \ar@{=}[d] & \coker\psi_0  \ar [d]^{\lambda_1u_1}  \\
 \ker\psi'_{0} \ar[r]^{\psi_1'} & \coker\psi'_0
}\]
then we have $\ker\psi_1=\ker\psi'_1$ and there is an induced isomorphism $u_2: \coker\psi_1\cong\coker\psi'_1$. Suppose this process continues, that is, assume $m=m'$ and there are isomorphisms 
\[ u_i: \coker\psi_{i-1}\cong\coker\psi'_{i-1}, \quad 1\le i\le m,
\]
and nonzero scalars $\lambda_i\in\bk$ such that $\psi'_i=\lambda_iu_i\psi_i$, 
and $u_{i+1}$ is induced from $u_{i}$ as above, then we say $(\sF, \psi_0,\cdots,\psi_m)$ and $(\sF',\psi'_0,\cdots,\psi'_{m'})$ are equivalent. Clearly, for nonzero scalars $c_0,\cdots,c_m$, $(\sF, \psi_0,\cdots,\psi_m)$ and $(\sF, c_0\psi_0,\cdots,c_m\psi_m)$ are equivalent. 

\medskip
 
There is another equivalent definition of complete stable pairs as follows. 

\begin{defn}\label{def_csp2}
A complete stable pair over $(\sF,\psi)$ with length $m$ consists of a sequence of coherent sheaves $\sF_0, \sF_1, \cdots, \sF_m$ on $X$ and homomorphisms 
\[ \varphi_i: V\otimes\sO_X\to \sF_i, \quad i=0,\cdots,m,
\] 
such that $(\sF_0,\varphi_0)=(\sF,\psi)$, and for each $i\ge 0$, $\sF_{i+1}$ fits into a short exact sequence
\begin{equation}\label{Fi}
0\to \coker\varphi_{i}\stackrel{\alpha_i}\lra\sF_{i+1}\stackrel{\beta_i}\lra \image\varphi_{i}\to 0
\end{equation}
satisfying $\beta_i\varphi_{i+1}=\varphi_{i}$, and the lengths of cokernels satisfy 
\[ \length(\coker\varphi_0)>\length(\coker\varphi_1)>\cdots>\length(\coker\varphi_m)=0.
\] 
\end{defn}

We denote such a complete stable pair by $(\sF_i,\varphi_i)_m$. From $\beta_i\varphi_{i+1}=\varphi_{i}$ in the definition, we see that $\beta_i$ induces a surjective homomorphisms
\begin{equation}\label{phi}
\image\varphi_{i+1}\to\image\varphi_i, \quad \coker\varphi_{i}\to \coker\varphi_{i+1}.
\end{equation} 
Since there are natural surjections $\sF_{i}\to\coker\varphi_{i}$ and injections $\image\varphi_{i}\to\sF_{i}$, the sequence \eqref{Fi} induces homomorphisms
\[ \sF_{i}\stackrel{\widetilde\alpha_i}\lra\sF_{i+1}\stackrel{\widetilde\beta_i}\lra\sF_{i}.
\] 

Two complete stable pairs $(\sF_i,\varphi_i)_m$ and $(\sF'_i,\varphi'_i)_{m'}$ are called equivalent if $m=m'$ and there are isomorphisms $\theta_i:\sF_i\to\sF_i'$ satisfying $\varphi_i'=\theta_i \varphi_i$ and 
\[ \theta_{i+1} \widetilde\alpha_i=\lambda_i\widetilde\alpha'_{i} \theta_{i},\quad \theta_{i} \widetilde\beta_i=\lambda'_i\widetilde\beta'_{i} \theta_{i+1}
\]
for some nonzero scalars $\lambda_i,\lambda'_i\in\bk$ and for all $i\ge 0$. Clearly, for nonzero scalars $c_0,\cdots,c_m$, $(\sF_i,\varphi_i)_m$ and $(\sF_i,c_i\varphi_i)_m$ are equivalent. 

Now we show the equivalence of Definition \ref{def_csp1} and \ref{def_csp2}. 

Let $(\sF, \psi_0,\cdots,\psi_m)$ be a complete stable pair on $X$. We define $(\sF_i,\varphi_i)_m$ inductively as follows. Let $\sF_0=\sF$ and $\varphi_0=\psi_0$. Suppose we have defined $\sF_i$ and $\varphi_i$ with $\ker\varphi_i=\ker\psi_i$ and an isomorphism $\coker\varphi_i\cong\coker\psi_i$, note that $\ker\psi_{i}$ is a subsheaf of $V\otimes \sO_X$, we define $\sF_{i+1}$ and $\varphi_{i+1}$ using the following pushout diagram 
\[ \xymatrix{
 \ker\psi_{i}\ar[r] \ar[d]^{\psi_{i+1}} & V\otimes \sO_X  \ar@{-->}[d]^{\varphi_{i+1}}  \\
 \coker\psi_{i} \ar@{-->}[r] & \sF_{i+1}  
}\] 
By $\ker\psi_{i}=\ker\varphi_{i}$, the above diagram can be extended as  
\begin{equation}\label{CD_Fi}
\xymatrix{
0 \ar[r]  &\ker\psi_{i}\ar[r] \ar[d]^{\psi_{i+1}} & V\otimes \sO_X \ar[r]^-{\varphi_{i}} \ar[d]^{\varphi_{i+1}} &\image\varphi_{i} \ar[r] \ar@{=}[d] & 0 \\
0\ar[r] &\coker\psi_{i} \ar[r] & \sF_{i+1} \ar[r]^{\beta_i} &\image\varphi_{i} \ar[r] & 0 
}\end{equation}
From $\coker\psi_{i}\cong\coker\varphi_{i}$, we see that $\sF_{i+1}$ satisfies the sequence \eqref{Fi} and $\beta_i\varphi_{i+1}=\varphi_{i}$. Moreover, by the snake lemma,  
 $\ker\psi_{i+1}=\ker\varphi_{i+1}$ and there is an induced isomorphism $\coker\psi_{i+1}\cong\coker\varphi_{i+1}$.  Since $\psi_i$ are nonzero, the lengths of $\coker\varphi_i$ decrease strictly.. 

Conversely, given a complete stable pairs of the form $(\sF_i,\varphi_i)_m$, let $\sF=\sF_0$ and $\psi_0=\varphi_0$. For $i\ge 0$, the condition $\beta_i\varphi_{i+1}=\varphi_{i}$ implies that we have the right square of commutative diagram \eqref{CD_Fi}. Taking kernels of $\varphi_i$ and $\beta_i$, and by the sequence \eqref{Fi}, we obtain homomorphisms $\psi_{i+1}: \ker\varphi_i\to\coker\varphi_i$. Finally, by snake lemma and induction on $i$, we have $\ker\psi_i=\ker\varphi_i$ and $\coker\psi_{i}\cong\coker\varphi_{i}$. In particular, $\coker\psi_{i}=0$ if and only if $\coker\varphi_{i}=0$. Therefore,  $(\sF, \psi_0,\cdots,\psi_m)$ is a complete stable pair. 

It is straightforward to verify that the correspondence preserves equivalence relations.

\subsection{Expanded stable pairs} 

We begin with some notations. Let $\gm=\spec \bk [t,t^{-1}]$. Define a linear $\gm$-action $\sigma$ on $\bP^1$ by
\[    \gm\times \bP^1\to \bP^1,\quad  (z_0, z_1)\mapsto (z_0, t^{-1}z_1). 
\]
Let $\bzero=(1,0)$ and ${\infty}=(0,1)$ be two fixed points on $\bP^1$. Let $\xi=(1,1)\in \bP^1$.  

Let $Y=X\times \bP^1$ with projections
\[ p:Y\to X,\quad  q: Y\to\bP^1.
\] 
Then the $\gm$-action $\sigma$ on $\bP^1$ induces a $\gm$-action on $Y$ by acting trivially on $X$. Let
\[ \iota^-:\,D_-=X\times \bzero\to Y, \quad \iota^+:\, D_+=X\times \infty\to Y
\] 
be the inclusions. We identify $D_-$ and $D_+$ with $X$ via the projection $p$. Let $\sO_Y(D_-)$ and $\sO_Y(D_+)$  be $\gm$-equivariant line bundles together with $\gm$-equivariant homomorphisms
\[ s^-:\,\sO_Y\to\sO_Y(D_-),\quad s^+:\,\sO_Y\to\sO_Y(D_+).
\]
The zero schemes of $s^-$ and $s^+$ are $D_-$ and $D_+$ respectively. Let 
\[ \sN_{-}=\coker s^-=\sO_{D_-}(D_-)\quad \text{and}\quad  \sN_{+}=\coker s^+=\sO_{D_+}(D_+)
\] 
be normal line bundles of $D_-$ and $D_+$ in $Y$ respectively. Then they are trivial bundles, and $\gm$ acts on the fibers of $\sN_{-}$ and $\sN_{+}$ with weight $-1$ and $1$ respectively. 

\medskip
  
Let $(\sF_i,\varphi_i)_m$ be a complete stable pair on $X$. We construct a $\gm$-equivariant sheaf $\widetilde \sF_i$ on $Y=X\times \bP^1$ together with a $\gm$-equivariant homomorphism $\widetilde\varphi_i: V\otimes \sO_{Y}\to \widetilde \sF_i$ for each $i$. For convenience, we denote 
\[ \sR_i=\image\varphi_i,\,\,\sT_i=\coker\varphi_i,  \,\,0\le i\le m, \quad\text{and}\quad \sR_{-1}=\sF_0,\,\,\sT_{-1}=0. 
\]
By Definition \ref{def_csp2}, for $0\le i\le m$, $\sF_i$ fits into exact sequences
\begin{align}  \label{Fi_1}
& 0\to\sR_i \to\sF_i \to\sT_i\to 0, \\
& 0\to \sT_{i-1}\to\sF_i\to\sR_{i-1}\to 0. \label{Fi_2} 
\end{align}
Consider the pullback of $\varphi_i:V\otimes \sO_X\to\sF_i$ to $Y$
\[ p^*\varphi_i:\, V\otimes \sO_Y\to p^*\sF_i.
\] 
Recall that  $s^-:\sO_{Y}\to\sO_{Y}(D_{-})$ is a $\gm$-equivariant homomorphism. Define 
\[ \widetilde\varphi_i=(p^*\varphi_i)\otimes s^-:\, V\otimes \sO_{Y}\to p^*\sF_i\otimes\sO_{Y}(D_{-}).
\]

We make elementary modifications of $p^*\sF_i\otimes\sO_{Y}(D_{-})$ as follows. First, note that the restriction of $p^*\sF_i\otimes\sO_{Y}(D_{-})$ to $D_+$ is naturally isomorphic to $\sF_i$. We let 
\[ \mu^+:\, p^*\sF_i\otimes\sO_{Y}(D_{-})\to \iota^+_* \sF_i  \to  \iota^+_* \sT_{i}
\]
be the composite of the restriction to $D_+$ and the homomorphism $\sF_i\to\sT_{i}$ in \eqref{Fi_1}. Then $\mu^+$ is a $\gm$-equivariant surjection on $Y$. Let 
\[ \sK_i=\ker\mu^+.
\] 
Next, consider the injection $\sK_i\to  p^*\sF_i\otimes\sO_{Y}(D_{-})$. Its restriction to $D_-$ is an isomorphism $\sK_i|_{D_-}\cong \sF_i\otimes \sN_-$. Composed with the homomorphism $\sF_i\to\sR_{i-1}$ in \eqref{Fi_2}, we obtain a $\gm$-equivariant surjection on $Y$ 
\[ \mu^-:\, \sK_i \to  \iota^-_*(\sF_i\otimes \sN_-) \to  \iota^-_*(\sR_{i-1}\otimes \sN_-).
\]  
Define 
\[ \widetilde \sF_i=\ker\mu^-. 
\]
Then we have $\gm$-equivariant subsheaves
$\widetilde \sF_i\subset \sK_i\subset p^*\sF_i\otimes\sO_{Y}(D_{-})$.  
 
\begin{lemm}\label{lemm_tFi}
The $\gm$-equivariant sheaf $\widetilde \sF_i$ is flat over $\bP^1$, and the image of $\widetilde\varphi_i$ lies in $\widetilde \sF_i$. Moreover, there are $\gm$-equivariant isomorphisms
\begin{align*}
&\theta_+: \widetilde \sF_i|_{D_+}\cong \sR_i\oplus (\sT_i\otimes\sN_+^\vee), \\
&\theta_-: \widetilde \sF_i|_{D_-}\cong \sR_{i-1}\oplus (\sT_{i-1}\otimes \sN_-). 
\end{align*}
The restrictions of $\widetilde\varphi_i: V\otimes \sO_{Y}\to \widetilde \sF_i$ to $D_+$ and $D_-$, when composed with $\theta_{\pm}$, are $(\varphi_{i},0)$ and $(\varphi_{i-1},0)$ respectively. 
\end{lemm}

\begin{proof}
By definition, $\sK_i$ fits into the following commutative diagram 
\begin{equation} \label{CD_Ki}
\xymatrix{
 0\ar[r] &  p^*\sR_{i} \otimes\sO_{Y}(D_{-})  \ar@{=}[r]\ar[d] &  p^*\sR_{i} \otimes\sO_{Y}(D_{-}) \ar[d] \\
 0\ar[r] &   \sK_i  \ar[r]  \ar[d] & p^*\sF_i \otimes\sO_{Y}(D_{-}) \ar[r]^-{\mu^+}   \ar[d] &    \iota^+_* \sT_{i}   \ar@{=} [d] \ar[r] &0 \\
 0\ar[r]  & p^*\sT_{i}\otimes\sO_{Y}(D_{-}) \otimes\sO_{Y}(-D_{+})  \ar[r]^-\tau   &  p^*\sT_{i}\otimes\sO_{Y}(D_{-}) \ar[r] &  \iota^+_* \sT_{i}  \ar[r] & 0 
}\end{equation}
where the middle column is indued from \eqref{Fi_1}. In the left column, the first and third terms are flat over $\bP^1$, so is $\sK_i$. 

Taking restriction of commutative diagram \eqref{CD_Ki} to $D_+$. Since $\tau|_{D_+}$ is a zero map, the restriction of $\sK_i \to p^*\sF_i\otimes\sO_{Y}(D_{-})$ to $D_+$ factors through a homomorphism 
\[ \sK_i|_{D_+}\to p^*\sR_{i} \otimes\sO_{Y}(D_{-})|_{D_+}=\sR_{i}.
\] 
It follows that the restriction of the left column to $D_+$ splits, which induces an isomorphism $\sK_i|_{D_+}\cong  \sR_{i}\oplus (\sT_{i}\otimes \sN_+^\vee)$. By the natural isomorphism $\sK_i|_{D_+}\cong \widetilde \sF_i|_{D_+}$ we get $\theta_+$. 

Next, note from \eqref{CD_Ki} that $p^*\sR_{i}\otimes\sO_{Y}(D_{-})  \subset\sK_i\subset p^*\sF_i\otimes\sO_{Y}(D_{-}) $. Composed with the homomorphism $p^*\sF_i\otimes\sO_{Y}(D_{-})\to p^*\sR_{i-1}\otimes\sO_{Y}(D_{-})$ induced from \eqref{Fi_2}, we obtain a homomorphism 
\[ \kappa: \sK_i\to p^*\sR_{i-1}\otimes\sO_{Y}(D_{-}).
\] 
By \eqref{phi}, $\sR_i\to \sR_{i-1}$ is surjective, and by \eqref{CD_Ki}, $p^*\sR_{i}\otimes\sO_{Y}(D_{-})$ is a subsheaf of $\sK_i$. It follows that  $\kappa$ is surjective.  Let $\sA_i$ be its kernel. We obtain  
\begin{equation} \label{Ki}
0\to\sA_i\to\sK_i\to p^*\sR_{i-1}\otimes\sO_{Y}(D_{-}) \to 0. 
\end{equation}
The restriction of \eqref{Ki} to $D_+$ is induced from \eqref{Fi_2}. By definition of $\widetilde \sF_i $, we have the following commutative diagram 
\begin{equation} \label{CD_Kip}
\xymatrix{
 0\ar[r] &  \sA_i  \ar[d]\ar@{=}[r] & \sA_i  \ar[d]  \\
 0\ar[r] &  \widetilde \sF_i  \ar[d]\ar[r]   & \sK_i  \ar[r]^{\mu^-}  \ar[d] &  \iota^-_*(\sR_{i-1}\otimes \sN_-) \ar@{=} [d] \ar[r] &0 \\
  0\ar[r]& p^*\sR_{i-1} \ar[r] & p^*\sR_{i-1} \otimes\sO_{Y}(D_{-})   \ar[r] &   \iota^-_*(\sR_{i-1}\otimes \sN_-)\ar[r] &0
}\end{equation}
In the middle column, since $\sK_i$ and $p^*\sR_{i-1} \otimes\sO_{Y}(D_{-})$ are flat over $\bP^1$, $\sA_i$ is also flat over $\bP^1$. It follows that in the left column, $\widetilde \sF_i$ is flat over $\bP^1$.  
 
In the same way, taking restriction of diagram \eqref{CD_Kip} to $D_-$ gives an isomorphism 
\[   \widetilde \sF_i|_{D_-}\cong \sA_i|_{D_-}\oplus \sR_{i-1}. 
\]
By the natural isomorphism $\sA_i|_{D_-}\cong\sT_{i-1}\otimes \sN_-$ we obtain $\theta_-$. 

Finally, by \eqref{CD_Ki}, we have $p^*\sR_i\otimes\sO_{Y}(D_-) \subset \sK_i$ and by \eqref{CD_Kip}, $\sK_i\otimes\sO_{Y}(-D_-)\subset \widetilde \sF_i $. It follows that  $p^*\sR_i\subset \widetilde \sF_i $. Therefore the image of $\widetilde\varphi_i$ lies in $\widetilde \sF_i $. The restrictions of $\widetilde\varphi_i$ to $D_+$ and $D_-$ are obvious. 
\end{proof}

\medskip
  
Let $(Y, D_+)$ be a pair. Recall that \cite{Li01}\cite{LW15} an expanded pair of $(Y, D_+)$ consists of a scheme $Y_{[m]}$ with simple normal crossings singularities
\[ Y_{[m]}= Y_0\cup Y_1\cup \cdots\cup Y_m 
\]
and a distinguished divisor $D_m\subset Y_{[m]}$, where each $Y_i$ is a copy of $Y$, and $\cup$ means that we glue $Y_i$ and $Y_{i+1}$ transversely along their divisors $D_+\subset Y_i$ and $D_-\subset Y_{i+1}$, the resulting identified Weil divisor is named $D_i$. Let $D_m$ be the divisor $D_+$ on $Y_m$. For convenience, we let $D_{-1}\subset Y_{[m]}$ be the divisor $D_-$ on $Y_0$.  Fix the $\gm$-action on $Y_{[m]}$ induced from the $\gm$-action on $Y$.   

Consider the pair $(\bP^1,\infty)$ and its expanded pair $(\bP^1_{[m]},\infty_m)$ where 
\[ \bP^1_{[m]}=\bP^1\cup \Delta_1\cup \cdots\cup \Delta_m
\]
with each $\Delta_i=\bP^1$ and $\infty_m$ is the point $\infty$ of $\Delta_m$. Then by $Y=X\times \bP^1$, we can identify $Y_{[m]}$ with $X\times \bP^1_{[m]}$ and the divisor $D_m$ with $X\times \infty_m$. 

\medskip

Given a complete stable pair $(\sF_i,\varphi_i)_m$ on $X$, we have constructed a $\gm$-equivariant homomorphism $\widetilde \varphi_i: V\otimes \sO_{Y_i}\to \widetilde \sF_i $ for each $i$. By Lemma \ref{lemm_tFi}, there is a $\gm$-equivariant isomorphism  
\[  \widetilde \sF_i |_{D_+}\cong \widetilde \sF_{i+1}|_{D_-}.
\]
Hence we can glue all these $\widetilde \sF_i$ along $D_i$ to obtain a $\gm$-equivariant sheaf $\widetilde \sF$ on $Y_{[m]}$. Moreover, since $\widetilde \varphi_{i}|_{D_+}$ and $ \widetilde \varphi_{i+1}|_{D_-}$ are identical as $\gm$-equivariant homomorphisms, $\widetilde \varphi_i$ can be glued along divisors $D_i$ to obtain a $\gm$-equivariant homomorphism 
\[  \widetilde\varphi: V\otimes \sO_{Y_{[m]}}\to \widetilde \sF.
\]
 
\begin{defn}
We say the $\gm$-equivariant sheaf $\widetilde \sF$ and $\gm$-equivariant homomorphism $\widetilde\varphi: V\otimes \sO_{Y_{[m]}}\to \widetilde \sF$ on $(Y_{[m]}, D_m)$ form an expanded stable pair $(\widetilde \sF, \widetilde\varphi)$. 
\end{defn}
 
Two expanded stable pairs $(\widetilde \sF, \widetilde\varphi)$ and $(\widetilde \sF', \widetilde\varphi')$ are equivalent if there is an automorphism $u$ of $(Y_{[m]},D_m)$ over $(Y,D_+)$, and an isomorphism $\alpha: \widetilde \sF\to u^*\widetilde \sF'$ so that the following diagram commutes
\[ \xymatrix{
V\otimes \sO_{Y_{[m]}}\ar[r]^-{\widetilde\varphi}\ar@{=}[d] &\widetilde \sF\ar[d]^{\alpha} \\
V\otimes \sO_{Y_{[m]}}\ar[r]^-{u^*\widetilde\varphi'} & u^*\widetilde \sF'
}\]
Here we identified $u^*(V\otimes \sO_{Y_{[m]}})$ with $V\otimes \sO_{Y_{[m]}}$ by sending the constant section $1$ of $\sO_{Y_{[m]}}$ to itself. 

By definition, equivalent complete stable pairs induce equivalent expanded stable pairs. 

\begin{prop}
There is a one-to-one correspondence between the set of equivalence classes of complete stable pairs with length $m$ on $X$ and the set of equivalence classes of expanded stable pairs on $(Y_{[m]},D_m)$. 
\end{prop}
\begin{proof}
For any complete stable pair $(\sF_i,\varphi_i)_m$ on $X$, we have associated an expanded stable pair $(\widetilde \sF, \widetilde\varphi)$ on $(Y_{[m]},D_m)$. To establish this one-to-one correspondence, it suffices to show that an expanded stable pair induces a complete stable pair which is the inverse of the previous construction. 

Let $(\widetilde \sF, \widetilde\varphi)$ be an expanded stable pair on $(Y_{[m]},D_m)$. 
Recall that $Y_{[m]}=X\times \bP^1_{[m]}$ with projection $q: Y_{[m]}\to \bP^1_{[m]}$.
Let $\xi_i=(1,1)\in \Delta_i$ which is distinct from $\bzero$ and $\infty$. Let 
\[ X_i= q^{-1}(\xi_i)
\]
Then $X_i$ is isomorphic to $X$ via the projection $p$. Define 
\[ \varphi_i: V\otimes \sO_{X_i}\to \widetilde \sF|_{X_i}
\]
be the restriction of $\varphi$ to $X_i$. Then $(\widetilde \sF|_{X_0},\varphi_0)$ is a stable pair on $X$, and $(\widetilde \sF|_{X_i}, \varphi_i)_m$ is a complete stable pair. Moreover, the expanded stable pair associated to $(\widetilde \sF|_{X_i}, \varphi_i)_m$ is equivalent to $(\widetilde \sF, \widetilde\varphi)$. 

Finally, it is easy to verify that the correspondence preserves equivalence relations. 
\end{proof}

\begin{prop}\label{prop_auto}
The group of automorphisms of an expanded stable pair is trivial. 
\end{prop}

\begin{proof}
Let $(\widetilde \sF, \widetilde\varphi)$ be an expanded stable pair on $(Y_{[m]},D_m)$. Since its restriction to $Y_0$ is a stable pair, it has trivial automorphism. The $\gm$-equivalence of $\widetilde\varphi$ and the surjectivity of $\widetilde\varphi|_{D_m}$ force the automorphism of $\widetilde\varphi$ being trivial.
\end{proof}

\subsection{Properties of expanded stable pairs}

Let $\sS$ be a $\gm$-equivariant sheaf on $Y$ which is flat over $\bP^1$. Suppose there are decompositions 
\begin{align}\label{wd-}
& \sS|_{D_-}=\sW_{-}^{0}\oplus \sW_{-}^{-1}, \\
& \sS|_{D_+}=\sW_{+}^{0}\oplus \sW_{+}^{-1} \label{wd+}
\end{align}
so that $\gm$ acts on $\sW_{\pm}^{i}$ with weight $i$. Let
$\sS|_{D_-}\to  \sW_{-}^{-1}$ and  $\sS|_{D_+} \to  \sW_{+}^{0}$
be the natural induced homomorphisms. Then we have a surjective homomorphism 
$\sS\to  \iota^-_*\sW_{-}^{-1}\oplus  \iota^+_*\sW_{+}^{0}$. Let 
\begin{equation}\label{em}
\sS^e=\ker(\sS\to  \iota^-_*\sW_{-}^{-1}\oplus  \iota^+_*\sW_{+}^{0}). 
\end{equation} 

\begin{defn}
Let $\sS$ be a $\gm$-equivariant sheaf on $Y$ which is flat over $\bP^1$. We say $\sS$ is admissible if it has decompositions \eqref{wd-} \eqref{wd+} and the elementary modification $\sS^e$ is $\gm$-equivariantly isomorphic to $p^*\sF\otimes \sO_Y(-D_+)$ for some sheaf $\sF$ on $X$. An admissible sheaf is trivial if it is isomorphic (not necessary $\gm$-equivariantly) to $p^*\sF$ for some sheaf $\sF$ on $X$. 
\end{defn}

\begin{lemm}\label{lemm_adm}
Let $\sS_1$ and $\sS_2$ be admissible sheaves on $Y$. Then $\sS_1\oplus \sS_2$ is admissible. Let $\varphi:\sS_1\to\sS_2$ be a $\gm$-equivariant homomorphism between admissible sheaves. Then $\varphi(\sS_1)$ is admissible.  
\end{lemm}
\begin{proof}
Follows from the definition. 
\end{proof}

\begin{lemm}\label{lemm_cons}
Let $\sS$ be an admissible sheaf on $Y$ with decompositions \eqref{wd-} and \eqref{wd+}. Denote $X_1=X\times \xi\subset Y$. Then there are induced exact sequences of sheaves on $X$
\begin{align}\label{seq+}
& 0\to \sW^0_+\stackrel{\gamma}\lra \sS|_{X_1}\to\sW^{-1}_+ \to 0, \\
& 0\to \sW^{-1}_- \to \sS|_{X_1}\stackrel{\beta}\lra\sW^{0}_-\to 0.\label{seq-}
\end{align}
Suppose further $\varphi:V\otimes \sO_Y\to\sS$ is a $\gm$-equivariant homomorphism, and let $\varphi_1=\varphi|_{X_1}$, $\varphi^{\pm}=\varphi|_{D_{\pm}}$. Then $\gamma\varphi^+=\varphi_1$ and $\beta\varphi_1=\varphi^-$. 
\end{lemm}

\begin{proof}
By definition, we have an exact sequence 
\begin{equation}\label{Fe}
 0\to p^*\sF\otimes \sO_Y(-D_+)\to \sS\to  \iota^-_*\sW_{-}^{-1}\oplus  \iota^+_*\sW_{+}^{0}\to 0.
\end{equation}
Restricting to $X_1$, it gives an isomorphism $\sF\cong \sS|_{X_{1}}$. Taking restriction of \eqref{Fe} to $D_+$, and using the flatness of $\sS$ over $\bP^1$, we have  
\[  0\to \sTor_1\to \sF\otimes \sN_+^\vee\to \sS|_{D_+}\to   \sW_{+}^{0}\to 0,
\]
where $\sTor_1=\sTor_1^{\sO_Y}(\sO_{D_+},  \iota^-_*\sW_{-}^{-1}\oplus\iota^+_*\sW_{+}^{0})\cong \sW_{+}^{0}\otimes\sO_Y(-D_+)\cong \sW_{+}^{0}\otimes\sN_+^\vee$.  By the decomposition \eqref{wd+}, we obtain a short exact sequence 
\[  0\to \sW_{+}^{0}\otimes\sN_+^\vee \to \sF\otimes \sN_+^\vee \to  \sW_{+}^{-1}\to 0.
\]
Using the isomorphism $\sN_+^\vee\cong \sO_{D_+}$ we obtain \eqref{seq+}. Similarly, the restriction of \eqref{Fe} to $D_-$ yields \eqref{seq-}. 

Now suppose there is a $\varphi: V\otimes \sO_Y\to \sS$. We obtain a commutative diagram 
\[ \xymatrix{
0\ar[r] &V\otimes \sO_{Y}(-D_+) \ar[r]^-{\tau}\ar[d] & V\otimes \sO_{Y} \ar[d]^{\varphi} \ar[r] &  V\otimes \sO_{D_+} \ar[r]\ar[d] & 0 \\
0\ar[r] & p^*\sF\otimes \sO_Y(-D_+) \ar[r]  & \sS\ar[r] & \iota^-_*\sW_{-}^{-1}\oplus  \iota^+_*\sW_{+}^{0} \ar[r] & 0
}\]
The left colume induces a homomorphism $V\otimes \sO_Y\to p^*\sF$. Restricting to $X_1$, we get a homomorphism $V\otimes \sO_{X_1}\to \sS|_{X_1}$ that is $\varphi_1$. On the other hand, restricting the diagram to $D_+$, and note that $\tau|_{D_+}$ is zero, we obtain a commutative diagram 
\[ \xymatrix{
0\ar[r] &V\otimes \sN_+^\vee \ar[d]^{\varphi^+\otimes 1} \ar@{=}[r] & V\otimes \sN_+^\vee \ar[d]^{\varphi_1\otimes 1} \ar[r]^0 &  V\otimes \sO_{D_+} \ar[d] \\
0\ar[r] & \sW_{+}^{0}\otimes\sN_+^\vee \ar[r]^{\gamma\otimes 1}  & \sF\otimes \sN_+^\vee \ar[r] &  \sS|_{D_+} 
}\]
It follows that $\gamma\varphi^+=\varphi_1$. The second identity can be obtained by restriction to $D_-$. 
\end{proof}

\begin{prop}\label{prop_cri}
Let $\sS$ be a $\gm$-equivariant sheaf on $Y_{[m]}$ which is flat over $\bP^1_{[m]}$. Let $\varphi:V\otimes \sO_{Y_{[m]}}\to \sS$ be a $\gm$-equivariant homomorphism. Let $m\ge 1$. Suppose 
\[ \sS|_{D_i}=\sW_i^{0}\oplus\sW_i^{-1}
\]
are decompositions. Then $(\sS,\varphi)$ is an expanded stable pair on $(Y_{[m]},D_m)$ if and only if it satisfies the following conditions:
\begin{enumerate}
\item $\sS|_{Y_i}$ are admissible for all $i$, and they are not trivial;
\item $\image(\varphi|_{D_i})=\sW_i^0$ for all $i$, and $\varphi|_{D_m}$ is surjective.
\end{enumerate} 
\end{prop}

\begin{proof}
The only if part is obvious from the definition. For the if part, it suffices to show that $\sS$ and $\varphi$ induce a complete stable pair on $X$ whose associated expanded stable pair is equivalent to $(\sS,\varphi)$. Let $X_i=X\times \xi_i$ where $\xi_i=(1,1)$ on $\Delta_i$. For convenience, denote
\[     \sF_i=\sS|_{X_i} \quad\text{and}\quad    \varphi_i=\varphi|_{X_i}. 
\]
Since $\sS|_{Y_i}$ are admissible, by Lemma \ref{lemm_cons}, there are exact sequences
\begin{align*} 
& 0\to \sW^0_i\stackrel{\gamma_i}\lra \sF_{i}\to\sW^{-1}_i \to 0, \\
& 0\to \sW^{-1}_{i-1} \to \sF_{i}\stackrel{\beta_i}\lra\sW^0_{i-1}\to 0, 
\end{align*}
so that $\varphi_i= \gamma_i(\varphi|_{D_i})$ and $\varphi|_{D_{i-1}}=\beta_i\varphi_i$. 
Let $\sR_i=\image \varphi_i$ and $\sT_i=\coker \varphi_i$. Then $\sR_i=\sW_i^0$ and it follows that $\sT_i\cong \sW_i^{-1}$. Moreover, since $\sS|_{Y_i}$ are not trivial, $\sR_i\ncong \sR_{i-1}$ for all $i$,  
Therefore, $(\sF_i, \varphi_i)_m$ is a complete stable pair on $X$. Its associated expanded stable pair is then $(\sS,\varphi)$. 
\end{proof}


\section{Moduli stack of expanded stable pairs}

\subsection{Families of expanded stable pairs}

Let $(Y,D_+)$ be a pair. Recall that the standard models
\[ \xymatrix{
(Y[n], D[n])\ar[r]^-{p_n}\ar[d]^{\pi_n} & (Y, D_+)  \\
\bA^n 
}\]
were constructed inductively as follows: $(Y[0], D[0])=(Y,D_+)$ and $p_0$ is the identity map;
for $n\ge 0$, $(Y[n+1], D[n+1])$ is the blowing-up of $Y[n]\times\bA^1$ along $D[n]\times 0$, and $D[n+1]$ is the strict transform of $D[n]\times \bA^1$; and $\pi_{n+1}, p_{n+1}$ are obtained by natural compositions. It is easy to see that the fiber of $\pi_n$ over $0\in \bA^n$ is the expanded pair $(Y_{[n]},D_n)$, together with a morphism $p_n: (Y_{[n]},D_n)\to (Y,D)$. 

Let $S$ be a $\bk$-scheme. A family of expanded pairs of $(Y,D)$ over $S$ consists of morphisms $\pi$ and $p$ 
\[ \pi: (\cY,\cD)\to S,\quad p: (\cY,\cD)\to (Y,D)
\]
so that there is an \'etale covering $\{S_\alpha\}$ of $S$ and morphisms $\eta_\alpha:S_\alpha\to \bA^{n_\alpha}$ for some $n_\alpha$, so that 
\[ (\cY_\alpha,\cD_\alpha):=(\cY,\cD)\times_S S_\alpha\cong \eta_\alpha^*(Y[m_\alpha], D[m_\alpha])
\] 
over $S_\alpha$ and compatible with their projections to $(Y,D_+)$. Since the $\gm$-action on $Y$ can be lifted naturally to a $\gm$-action on $(Y[m_\alpha], D[m_\alpha])$, it induces a $\gm$-action $\sigma_\alpha$ on $(\cY_\alpha,\cD_\alpha)$
over $S_\alpha$ for each $\alpha$. These actions are compatible. So we obtain a $\gm$-action on $(\cY,\cD)\to S$ which acts trivially on $S$ and acts on the fibers in the same way as before. 

For any $\bk$-scheme $S$ and an expanded pair $(\cY,\cD)$ over $S$,   $V\otimes  \sO_{\cY}$ admits a $\gm$-action induced from the $\gm$-action on $\cY$ and the trivial action on $V$.

\begin{prop}\label{prop_open}
Let $(\cY,\cD)$ be an expanded pair over $S$. Let $\sF$ be an $S$-flat family of $\gm$-equivariant sheaf on $(\cY,\cD)$ and $\varphi: V\otimes \sO_{\cY}\to\sF$ be a $\gm$-equivariant homomorphism. Let $S^\circ\subseteq S$ be the subset consisting of those closed points $s\in S$ so that $\varphi_s: V\otimes \sO_{\cY_s}\to \sF_{s}$ is an expanded stable pair on $(\cY_s,\cD_s)$. Then $S^\circ$ is a Zariski open subset of $S$. 
\end{prop}

\begin{proof}
We assume $S$ is reduced and irreducible without loss of generality. Let 
\begin{equation}\label{decY}
\cY=\cY_0\cup\cY_1\cup\cdots\cup \cY_m
\end{equation} 
be the irreducible decomposition in which each $\cY_i$ is irreducible. Let $\cD_i=\cY_i\cap\cY_{i+1}$. Let $(\sF_s,\varphi_s)$ be an expanded stable pair on $(\cY_s,\cD_s)$. It suffices to show that there is a Zariski open subset $U\subset S$ containing $s$ so that for every point $t\in U$, $(\sF_t,\varphi_t)$ is an expanded stable pair on $(\cY_t,\cD_t)$. 

For simplicity, we assume that in the decomposition \eqref{decY}, 
\[ (\cY_{i}\times_S t,\cD_{i}\times_S t)\cong (Y,D_+)
\] 
for all $i$ and all $t\in S- \{s\}$. The general case can be treated in a similar way by induction. 

Note that $\sF|_{\cD_i}$ is flat over $S$, and it admits a $\gm$-action. There is a decomposition 
\[ \sF|_{\cD_i}=\mathop\oplus\limits_{k=-\infty}^{\infty} \sW_{i}^k 
\]  
so that $\gm$ acts on $\sW_{i}^k$ with weight $k$. It follows that $\sW_i^k$ are flat over $S$ as well. Since $(\sF_s,\varphi_s)$ is an expanded stable pair, the restriction $\sW_i^k|_{\cY_s}=0$ for $k\ne 0, -1$. By the $S$-flatness of $\sW_i^k$, we obtain $\sW_i^k=0$ for $k\ne 0, -1$. 

Let $\sF_i=\sF|_{\cY_i}$ and $\varphi_i=\varphi|_{\cY_i}$. Now we perform elementary modifications to $\sF_i$ as in \eqref{em}
\[ \sF_i^e=\ker(\sF_i\to  \iota^-_*\sW_{i-1}^{-1}\oplus  \iota^+_*\sW_{i}^{0}). 
\] 
By the $S$-flatness of $\sW_i^k$, $\sF_i^e$ is flat over $S$. 
Since $(\sF_s,\varphi_s)$ is an expanded stable pair, its restriction to all irreducible components of $\cY_i\times_S s$ are admissible. From the weights of the $\gm$-action, we see that the restriction $\sF_i|_{\cY_{i}\times_S t}$ are admissible. Finally, since surjectivity of a homomorphism is an open condition, we can find a  Zariski open subset $U\subset S$ containing $s$ so that condition (2) in Proposition 1.10  holds true for $\varphi_i|_{\cY_i\times_S t}:  V\otimes \sO_{\cY_i\times_S t}\to \sF_i|_{\cY_{i}\times_S t}$ when $t\in U$. In conclusion, $(\sF_t,\varphi_t)$ is an expanded stable pair on $(\cY_t,\cD_t)$ when $t\in U$. 
\end{proof}

Based on this proposition, we have
 
\begin{defn}
Let $\varphi: V\otimes \sO_{\cY}\to\sF$ be a $\gm$-equivariant homomorphism on $(\cY,\cD)\to S$. We say $(\sF, \varphi)$ is a family of expanded stable pairs if for every closed point $s\in S$, $(\sF_s, \varphi_s)$ is an expanded stable pair on $(\cY_s,\cD_s)$. 
\end{defn}

\subsection{Stack of expanded stable pairs}

We define the category $\fP_{X}^h(V)$ of expanded stable pairs as follows. An object $(\sF,\varphi; \cY, \cD, S)$ in $\fP_{X}^h(V)$ consists of an expanded pair $(\cY, \cD)$ over a $\bk$-scheme $S$ and an expanded stable pair $\varphi: V\otimes \sO_{\cY}\to\sF$ on $(\cY,\cD)$ such that for any closed point $s\in S$ the restriction $\sF_s|_{\cD_s}$ has Hilbert polynomial $h$. An arrow between two objects
\[ (\sF_1,\varphi_1; \cY_1, \cD_1, S_1)\to (\sF_2,\varphi_2; \cY_2,\cD_2, S_2)
\] 
consists of a Cartesian square 
\[ \xymatrix{
(\cY_1,\cD_1)\ar[r]^f \ar[d] &(\cY_2,\cD_2)\ar[d] \\
S_1\ar[r] & S_2
}\]
and an isomorphism of expanded stable pairs $f^*\varphi_2\cong \varphi_1$, that is, a commutative diagram 
\[ \xymatrix{
V\otimes \sO_{\cY_1}\ar[r]^-{f^*\varphi_2}\ar@{=}[d] & f^*\sF_2 \ar[d]_{\cong}^u \\
V\otimes \sO_{\cY_1}\ar[r]^-{\varphi_1} & \sF _1
}\]
It is straightforward to verify that $\fP_{X}^h(V)$ is a groupoid. 

\begin{prop}
The groupoid $\fP_{X}^h(V)$ is an algebraic space of finite type.
\end{prop}

\begin{proof}
It is standard to verify that  $\fP_{X}^h(V)$ is a stack. Now we show that $\fP_{X}^h(V)$ is bounded.  Let $A_{X,h}$ be the set of equivalence classes of expanded stable pairs 
\[ \varphi:V\otimes \sO_{Y_{[m]}}\to \sF
\] 
with Hilbert polynomial $h$. Then the restriction of $(\sF,\varphi)$ to the divisor $D_{-1}\subset Y_{[m]}$ 
\[ \varphi|_{D_{-1}}: V\otimes\sO_{D_{-1}}\to \sF|_{D_{-1}}
\] 
is contained in $\M_X^h(V)$. Note that $m$ is bounded by the length of $\coker(\varphi|_{D_{-1}})$. It follows that $A_{X,h}$ is bounded.  

Next we show that $\fP_{X}^h(V)$ is locally a stack quotient of a scheme and hence algebraic. Fix any integer $n$. Let $\sL$ be a relative ample line bundle on the standard model $\pi: Y[n]\to\bA^n$. For any sheaf $\sF$ on $Y[n]$ which is flat over $\bA^n$ and any integer $m$, let 
\[ \sF(m)=\sF\otimes \sL^{\otimes m}. 
\]
Now since $A_{X,h}$ is bounded, there is an integer $m_0$ so that for all $m\ge m_0$, 
\[   P(m):= \chi(\sF_s(m))=h^0(\sF_s(m))
\] 
for any $\sF_s\in A_{X,h}$. 

Fix such an $m_0$ and let $W$ be a $\bk$-vector space with dimension $P(m_0)$. Let $R_n$ be the set of homomorphisms $\varphi:V\otimes \sO_{Y[n]}\to \sF
$ where $\sF$ is a sheaf on $Y[n]$ that is flat over $\bA^n$ with Hilbert polynomial $P$. 
Then for $\varphi\in R_n$, we have a commutative diagram
\[ \xymatrix{
 \pi^*\pi_*(V\otimes \sO_{Y[n]}(m_0))\ar[d] \ar[r]^-{\varphi'}&   \pi^*\pi_*(\sF(m_0))\cong  W\otimes \sO_{Y[n]}  \ar[d]^\rho \\
 V\otimes \sO_{Y[n]}(m_0)\ar[r]^-{\varphi(m_0)}&  \sF(m_0) 
}\]

Let $\Quot_{Y[n]/\bA^n}^{P}(W)$ be the relative Quot-scheme parameterizing quotients
\[  \rho: W\otimes \sO_{Y[n]}  \to \sF(m_0)
\] 
over $\bA^n$ so that the Hilbert polynomial of $\sF_s$ is $P$ for any $s\in \bA^n$.  Then we get an injective map
\begin{equation}\label{str}
 R_n\to \Quot_{Y[n]/\bA^n}^{P}(W)\times \Hom(\pi_*(V\otimes \sO_{Y[n]}(m)), W\otimes \sO_{\bA^n}),\quad \varphi \mapsto (\rho,\varphi')
\end{equation}
whose image admits a natural subscheme structure which is still denoted by $R_n$. Moreover, the induced morphism $R_n\to \bA^n$ defines an expanded pair $(\cY,\cD)\to R_n$ and there is a tautological family of homomorphisms on $(\cY,\cD)$
\[ \widetilde\varphi: V\otimes \sO_{\cY}\to \widetilde \sF.
\]

There is a natural $\PGL(W)$-action on the right hand side of \eqref{str}. It also admits a $\gmn$-action induced from $Y[n]$. Let $M_n\subset R_n$ be the subset consisting of expanded stable pairs. By Proposition \ref{prop_open}, $M_n$ admits a natural open subscheme structure. Moreover, $M_n$ is preserved by the $\PGL(W)$-action and the $\gmn$-action. Therefore, the tautological family over $M_n$ induces a morphism
\[ f_n: [M_n/( \PGL(W)\times \gmn)]\to  \fP_{X}^h(V).
\]
Hence we have a morphism 
\[ \coprod_{n\ge 0} f_n: \coprod_{n\ge 0}[M_n/( \PGL(W)\times \gmn)]\to  \fP_{X}^h(V)
\]
that is \'etale and surjective by construction. Finally, Proposition \ref{prop_auto} shows that $\fP_{X}^h(V)$ is an algebraic space. This completes the proof. 
\end{proof}

Recall that on $(Y,D_+)$, there are two distinguished divisors, $D_+$ and $D_-$. Similarly, on each expanded pair $(\cY,\cD)$ over $S$, there are two distinguished divisors, $\cD_{-1}$ and $\cD$. Let $(\sF, \varphi)$ be a family of expanded stable pairs on $(\cY,\cD)$. Then the restriction to $\cD_{-1}$
\[ \varphi|_{\cD_{-1}}: V\otimes \sO_{\cD_{-1}}\to\sF|_{\cD_{-1}}
\]
is a family of stable pairs on $X$ parameterized by $S$, and the restriction to $\cD$
\[
 \varphi|_{\cD}: V\otimes \sO_{\cD}\to\sF|_{\cD}
\] 
is a family of quotients on $X$ parameterized by $S$. Therefore, we obtain natural morphisms
\[ \fP_{X}^h(V)\to \M_{X}^h(V),\quad \fP_{X}^h(V)\to \Quot_{X}^h(V).
\]

\subsection{Properness of the moduli stack} 
We prove the properness of $\fP_{X}^h(V)$ by valuative criterion. Let $\Delta$ be a nonsingular affine curve with a closed point $0\in\Delta$; let $\Delta^*=\Delta- \{0\}$. 

Let $(\cY^*,\cD^*)\to \Delta^*$ be a family of expanded pairs over $\Delta^*$, and let
$\varphi^\star: V\otimes \sO_{\cY^*}\to \sS^\star$ be a family of expanded stable pairs. We will show that possibly after base change, there exists a family of expanded pairs $(\cY, \cD)$ over $\Delta$ so that $(\cY,\cD)\times_\Delta \Delta^*=(\cY^*,\cD^*)$, together with a family of expanded stable pairs $ \varphi: V\otimes \sO_{\cY}\to \sS$ extending $\varphi^\star$.

By the construction of stack of expanded pairs, possibly after base change if necessary, we have $(\cY^*,\cD^*)=(Y_{[m]},D_m)\times \Delta^*$ for some $m$. For simplicity, we assume $m=0$, the general case $m>0$ is similar and we omit the details. Now we have $(\cY^*,\cD^*)=(Y,D_+)\times \Delta^*$. Let 
\[  (\cY^{0},\cD^{0})=(Y,D_+)\times \Delta. 
\]
Let $p:Y\times \Delta^*\to X\times  \Delta^*$ be the projection. By definition, $\varphi^\star$ is the pullback $p^*\psi^\star$ of a family of surjective stable pairs $\psi^\star:V\otimes \sO_{X\times \Delta^*}\to \sF^\star$ on $X\times \Delta^*$ . By the properness of the moduli functor of stable pairs on $X$, we can find a family of stable pairs $\psi:V\otimes \sO_{X\times \Delta}\to \sF$ extending $\psi^\star$. Let 
\[ \sS^0=p^*\sF,\quad \varphi^{0}=p^*\psi: V\otimes \sO_{\cY^0}\to\sS^0
\]
where $p: \cY^{0}\to X\times \Delta$ is the projection. We denote the restriction of $(\sF,\psi)$ to the central fiber $X\times 0$ by 
\[  (\sF_0, \psi_0)=(\sF|_{X\times 0}, \psi|_{X\times 0}).
\]

If $\psi_0$ is surjective, we let $(\cY,\cD)=(\cY^{0},\cD^{0})$ and $(\sS,\varphi)=(\sS^0,\varphi^0)$. It is a family of expanded stable pairs and we are done. Otherwise, the restriction of $\varphi^0$ to the central fiber $(\cY^{0},\cD^{0})\times_{\Delta}{0}$ is not an expanded stable pair. Now we let $\pi:\cY^{1}\to\cY^{0}$ be the blowing-up of $\cY^{0}$ along $D_+\times 0$, and let $\cD^1$ be the strict transform of $\cD^0$. Then we get a family of expanded pairs $(\cY^{1},\cD^{1})\to \Delta$  
with central fiber 
\[   (\cY^{1},\cD^{1})\times_\Delta{0}=(Y_{[1]}=Y_0\cup Y_1, D_1).
\]
Let 
\[ \sS'=\pi^*\sS^0,\quad \varphi'=\pi^*\varphi^0: V\otimes \sO_{\cY^{1}} \to  \sS'.
\]
Then clearly, the restriction of $\varphi'$ to the central fiber $Y_{[1]}$ of $\cY_1$ is the pullback of the stable pair $(\sF_0,\psi_0)$ on $X$. Because $\psi_0$ is not surjective, it fits into a short exact sequence 
\[  0\to\image\psi_0\to \sF_0\stackrel{\gamma}\to \coker\psi_0\to 0
\]
with $\coker\psi_0\ne 0$. Let $p_1:Y_1\to X$ be the projection. Consider the natural surjection 
\[ \mu: \sS'\to \sS'|_{Y_1}=p_1^*\sF_0\stackrel{p_1^*\gamma}\lra p_1^*\coker\psi_0. 
\]
Let $\sS^1=\ker\mu$. Then there is a commutative diagram 
\begin{equation}\label{}
\xymatrix{
0 \ar[r] & V\otimes\sO_{\cY^{1}} \ar@{=}[r]\ar[d]^{\varphi^1} & V\otimes\sO_{\cY^{1}}  \ar[d]^{\varphi'}\ar[r] & 0 \ar[d]\ar[r] & 0 \\
0 \ar[r] &\sS^1\ar[r] & \sS' \ar[r]^-{\mu} & p_1^*\coker\psi_0 \ar[r] & 0
}\end{equation}
Consider the restriction of $\varphi^1$ to the central fiber $(Y_{[1]}, D_1)$ of $(\cY^1,\cD^1)$. If the restriction of $\varphi^1$ to the divisor $D_1\subset Y_{[1]}$ is surjective, then it is easy to verify that  $(\sS^1,\varphi^1)$ is a complete stable pairs. Otherwise, we continue this process to obtain a family of expanded pairs $(\cY^{2},\cD^{2})\to \Delta$ and a homomorphism 
$\varphi^2: V\otimes \sO_{\cY^2}\to \sS^2$. Because the length of $\coker\varphi^i$ decreases, after finite steps, this process will terminate and we obtain a family of expanded stable pairs extending $\varphi^\star$.  

Next we show the uniqueness of extension. Suppose there are two flat families of expanded stable pairs 
\[ \varphi_j: V\otimes \sO_{\cY_j}\to \sS_j,\quad j=1,2
\]
such that they are isomorphic over $\Delta^*$. We show that this isomorphism extends over $0\in \Delta$. Again we assume after base change if necessary, $(\cY_j,\cD_j)\times_\Delta \Delta^*=(Y,D_+)\times \Delta^*$. By the construction of expanded pairs, we can find a third expanded pair $(\cY,\cD)\to\Delta$ and dominant morphisms 
\[ \pi_1: (\cY,\cD)\to (\cY_1,\cD_1),\quad \pi_2: (\cY,\cD)\to (\cY_2,\cD_2)
\] 
so that they are isomorphisms over $\Delta^*$. Pulling back $\varphi_1$ and $\varphi_2$ to $(\cY,\cD)$ respectively, we obtain two families of homomorphisms 
\[\widetilde\varphi_j: V\otimes \sO_{\cY}\to \widetilde\sS_j:=\pi_j^*\sS_j,\quad j=1,2
\] 
and will show that $\widetilde\varphi_1\cong\widetilde\varphi_2$. 

Recall that $\widetilde \varphi_1^\star\cong\widetilde \varphi_2^\star$ is the pullback $p^*\psi^\star$ of a family of surjective stable pairs $(\sF^\star, \psi^\star)$ on $X\times \Delta^*$ . Since an expanded stable pair on $Y=X\times \bP^1$ can be viewed as an ordinary stable pair with two dimensional sheaf on $Y$, $\widetilde \varphi_j^\star=p^*\psi^\star$ is a family of ordinary stable pairs parameterized by $\Delta^*$.  By the properness of moduli of stable pairs, there is a unique family of stable pairs $\psi:V\otimes \sO_{X\times \Delta}\to \sF$ extending $\psi^\star$, and 
\[ p^*\psi: V\otimes \sO_{\cY}\to p^*\sF
\]
is the unique family of stable pairs extending $p^*\psi^\star$, where $p: \cY\to X\times \Delta$ is the projection.

Using the procedure of stable reduction, we see that 
\[ \widetilde \sS_j\subseteq p^*\sF
\] 
and $\widetilde\varphi_j$ are compatible with $p^*\psi$. We form 
\[ \widetilde\varphi_1+\widetilde\varphi_2: V\otimes \sO_{\cY}\to\widetilde \sS_1+\widetilde \sS_2.
\] 
By Lemma \ref{lemm_adm} and Proposition \ref{prop_cri}, it is again a family of expanded stable pairs that is isomorphic to $\varphi_j$ over $\Delta^*$. Therefore, without loss of generality, we assume $\widetilde \sS_1\subseteq \widetilde \sS_2$. Let 
\[ (\cY,\cD)\times_\Delta 0=(Y_0\cup Y_1\cup\cdots\cup  Y_n, D_n)
\]  
be the central fiber of $(\cY,\cD)\to \Delta$. It induces homomorphisms 
\[ \widetilde \sS_1|_{D_i}\to \widetilde \sS_2|_{D_i},\quad i=-1,0,\cdots, n
\]
that is compatible to $\widetilde\varphi_1|_{D_i}$ and $\widetilde\varphi_2|_{D_i}$. Note that the restriction of $\widetilde\varphi_i$ to $D_{-1}$ are stable pairs. By the separatedness of the moduli functor of stable pairs, we have $\widetilde\varphi_1|_{D_{-1}}\cong\widetilde\varphi_2|_{D_{-1}}$. On the other hand, $\widetilde\varphi_j|_{D_n}$ are both surjective, and hence $\widetilde\varphi_1|_{D_{n}}\cong\widetilde\varphi_2|_{D_{n}}$. By the property of expanded stable pairs, we must have $\widetilde\varphi_1\cong\widetilde\varphi_2$. 

Having $\widetilde\varphi_1\cong\widetilde\varphi_2$, we let $I_1\subset \{0,1,\cdots,n\}$ be the set of indices so that the components $Y_i$ for $i\in I_1$ are collapsed by the morphism $\pi_1$. Likewise, define the subset $I_2\subset \{0,1,\cdots,n\}$ for the morphism $\pi_2$. We claim that $I_1=I_2$, for otherwise, there is a component $Y_j$ on which the sheaf 
\[ \widetilde\sS_1|_{Y_j} \cong \widetilde\sS_2|_{Y_j} 
\] 
is a trivial admissible sheaf which contradicts to an expanded stable pair. So we conclude that the families $(\cY_i,\cD_i)$ are isomorphic to $(\cY,\cD)$, and hence the families $\varphi_i$ are isomorphic to  $\widetilde\varphi_i$. Therefore, $\varphi_1\cong\varphi_2$ over $\Delta$.


\section{Moduli space of complete stable pairs on $\bP^1$}
 
Let $\M^{r,n}_{\bP^1}(V)$ be the moduli space of stable pairs $(\sF,\psi)$ on $\bP^1$ with $\rk(\sF)=r$ and $\deg(\sF)=n$. Then $\M^{r,n}_{\bP^1}(V)$ is a smooth projective variety over $\bk$. In this section we let $r=1$ and $N=\dim V\ge 2$. The higher rank case need some modifications. Let $\fP^{1,n}_{\bP^1}(V)$ be the moduli space of complete stable pairs on $\bP^1$ with rank $1$ and degree $n$. We will show that  $\fP^{1,n}_{\bP^1}(V)$ is an iterated blowing-up of $\M^{1,n}_{\bP^1}(V)$ along smooth centers. 

For simplicity, we denote 
\[  M(n)=\M^{1,n}_{\bP^1}(V).
\] 
For any stable pair $(\sF, \psi)$ in $M(n)$, we have $\sF\cong\sO_{\bP^1}(n)$ and $\psi$ is a nonzero homomorphism. With the notation \[ W_n=H^0(\sO_{\bP^1}(n))
\] 
we have an isomorphism
\[ M(n)\cong \bP(V^*\otimes W_n)\cong\bP^{(n+1)N-1}.
\] 

For any $j=0,\cdots, n$, define the subset $M(n)_j\subseteq M(n)$ as
\[ M(n)_j=\{(\sO_{\bP^1}(n), \psi)\in M(n)\mid  \deg(\image \psi) \le j \}.
\] 
Then $M(n)_j$ is a Zariski closed subset of $M(n)$, and      
\[ M(n)_0\subset M(n)_1\subset \cdots\subset M(n)_{n-1}\subset M(n)_n=M(n).
\] 
It is obvious that $M(n)_0$ is isomorphic to $\bP(V^*)\times \bP(W_n)$, and the inclusion $M(n)_0\subset M(n)$ is the Segre embedding.  

Now we describe ideal sheaves of $M(n)_j$ in $M(n)$. 

Let $m$ be a positive integer.  For any stable pair $(\sO_{\bP^1}(n), \psi)$ in $M(n)$,  tensoring with $\sO_{\bP^1}(m)$, we obtain a homomorphism 
\[ \psi(m):\, V\otimes \sO_{\bP^1}(m) \to \sO_{\bP^1}(n+m).
\] 
Taking global sections of $\psi(m)$, we get a linear map 
\[ \Gamma_{\psi(m)}: V\otimes W_m\to W_{n+m}.
\]    
Therefore, we obtain a morphism 
\[ f_m:\,M(n)\to S=\bP\Hom(V\otimes W_m, W_{n+m}),\quad \psi\mapsto \Gamma_{\psi(m)}.
\]
There is an integer $m_0$ so that for all $m>m_0$, all sheaves $\ker\psi(m)$, $\image\psi(m)$ have vanishing higher cohomologies for all $\psi$ in $M(n)$; moreover, $f_m$ is a closed immersion, and for each $(\sO_{\bP^1}(n), \psi)$ in $M(n)$,  the length
\[ \length(\coker\psi)=\dim\coker \Gamma_{\psi(m)}.
\] 

There is a tautological sub-line bundle on $S$
\[ \sO_S(-1)\to \Hom(V\otimes W_m, W_{n+m})\otimes \sO_S
\]
corresponding to a nowhere zero universal homomorphism 
\[ u: V\otimes W_m\otimes \sO_S\to W_{n+m}\otimes\sO_S(1).
\]
Let $Z_r$ be the scheme of zeros of $\wedge^{r+1}u$ with ideal sheaf $\sI_{Z_r}$. Then $Z_0=\emptyset$, and points in $Z_r$ are linear maps of  rank at most $r$. Let 
\[ E_{r}= \Hom (\mathop \bigwedge^{r+1}(V\otimes W_m), \mathop \bigwedge^{r+1}W_{n+m}). 
\]
Then $u$ induces a homomorphism
\[ u_r: E_r\otimes \sO_S(-r-1)\to \sO_S
\]
and  $\sI_{Z_r}=\image u_r$. In conclusion, we have

\begin{lemm}\label{lemm_ideal}
Fix a large enough integer $m$.  Then $(\sO_{\bP^1}(n), \psi)\in M(n)_j$ if and only if the linear map $\Gamma_{\psi(m)}$ is contained in $Z_{m+1+j}$. Therefore, $M(n)\times_S Z_{m+1+j}=M(n)_j$, and the ideal sheaf of $M(n)_j$ in $M(n)$ is $\sI_{M(n)_j}=f_m^{-1}\sI_{Z_{m+1+j}}\cdot \sO_{M(n)}$.
\end{lemm}

Next we construct varieties $M(n)^{(i)}$ as successive blowing-ups of $M(n)$ for $0\le i\le n$. Let  
\[ M(n)^{(0)}=M(n),\quad  M(n)_j^{(0)}=M(n)_j, \,\, j=0,\cdots,n.  
\]
Suppose we have defined $M(n)^{(k)}$ and $M(n)^{(k)}_j$, $0\le j\le n$, we let 
\[ \pi_{k+1}: M(n)^{(k+1)}\to M(n)^{(k)}
\] 
be the blowing-up of $M(n)^{(k)}$ along $M(n)_{k}^{(k)}$. Let $M(n)_{k}^{(k+1)}$ be the inverse image of $M(n)_{k}^{(k)}$. It is the exceptional divisor of $\pi_{k+1}$. For any $j\ne k$, let $M(n)_j^{(k+1)}$ be the strict transform of $M(n)_j^{(k)}$. It follows that for $j\le k$, $M(n)_{j}^{(k+1)}$ are divisors of $M(n)^{(k+1)}$, and there are inclusions of subvarieties
\[ M(n)^{(k+1)}_{k+1}\subset   \cdots\subset M(n)^{(k+1)}_{n-1}\subset M(n)^{(k+1)}.
\] 
In particular, $M(n)^{(n)}$ is a projective variety, together with $n$ distinguished divisors 
\[ M(n)_0^{(n)}, \,M(n)_1^{(n)},\, \cdots, \,M(n)_{n-1}^{(n)}.
\] 

\medskip

\begin{thm}\label{thm_isom}
There is an isomorphism 
\[ \fP^{1,n}_{\bP^1}(V)\cong M(n)^{(n)}. 
\]
\end{thm}

\begin{proof}
Fix a large enough integer $m$. Recall that $S=\bP\Hom(V\otimes W_m, W_{n+m})$. Let $S^{(n+m)}$ be the space of complete collineations from $V\otimes W_m$ to $W_{n+m}$. It is an iterated blowing up of $S$. 

Define a map $\fP^{1,n}_{\bP^1}(V)\to S^{(n+m)}$ as follows. Let $(\sO_{\bP^1}(n),\psi_0,\cdots,\psi_l)$ be a complete stable pair. Taking global sections of $\psi_i(m)$, we obtain a sequence of linear maps
\begin{align*}
&\Gamma_{\psi_0(m)}:  V\otimes W_m\to W_{n+m}, \\
&\Gamma_{\psi_1(m)}:  H^0(\ker \psi_0(m))  \to H^0(\coker\psi_0(m)), \\
&\hskip 50pt \vdots  \\
&\Gamma_{\psi_l(m)}:  H^0(\ker \psi_{l-1}(m))  \to H^0(\coker\psi_{l-1}(m)).
\end{align*}
By the choice of $m$, $\ker \psi_i(m)$ and $\coker\psi_i(m)$ have vanishing higher cohomologies. It follows that $\{\Gamma_{\psi_i(m)}\}_{i=0,\cdots,l}$ is a complete collineation from $V\otimes W_m$ to $W_{n+m}$, i.e., it is an element in $S^{(n+m)}$. Therefore, we have a commutative diagram 
\[ \xymatrix{
\fP^{1,n}_{\bP^1}(V) \ar[r] \ar[d] & S^{(n+m)} \ar[d] \\
M(n) \ar[r] & S
} 
\]
Since $M(n)^{(n)}=M(n)\times_S S^{(n+m)}$, we obtain a morphism $F:\fP^{1,n}_{\bP^1}(V)\to M(n)^{(n)}$. 

Conversely, there is a universal stable pair on $\bP^1\times M(n)$. Its pullback to $\bP^1\times M(n)^{(n)}$ is again a family of stable pairs. Similar to the proof of properness of $\fP^{1,n}_{\bP^1}(V)$, the blowing-up construction induces a family of expanded stable pairs on an expanded pair parametrized by $M(n)^{(n)}$, and it induces a morphism $G: M(n)^{(n)}\to \fP^{1,n}_{\bP^1}(V)$. It is standard to verify that $G$ is the inverse of $F$, and hence they are isomorphisms. 
\end{proof}

\begin{prop}\label{prop_isom}
Let $\Hilb^{m}$ be the Hilbert scheme of $m$ points on $\bP^1$. For any $0\le k\le n$, there is an isomorphism 
\begin{equation}\label{isom}
f: \,  M(n)_k^{(k)} \cong M(k)^{(k)}\times \Hilb^{n-k}.
\end{equation} 
so that for any $j\le k$, $f(M(n)^{(k)}_j)=M(k)_j^{(k)}\times\Hilb^{n-k}$. 
\end{prop}

We sketch the idea of the proof. 
For any $k$, let $(\sO_{\bP^1}(k), \psi)$ be a stable pair in $M(k)$. Let $s: \sO_{\bP^1}(k)\to\sO_{\bP^1}(n)$ be a nonzero homomorphism. Then $(\sO_{\bP^1}(n), s \psi)$ is a stable pair in $M(n)$. Since there is an isomorphism $\Hilb^{n-k}\cong \bP\Hom(\sO_{\bP^1}(k),\sO_{\bP^1}(n))$, we obtain a map
\[ g:\, M(k)\times \Hilb^{n-k}\to M(n),\quad (\psi, s)\mapsto s\psi, 
\]
whose image is $M(n)_k$, and it induces an isomorphism 
\[  (M(k)- M(k)_{k-1}) \times \Hilb^{n-k} \cong M(n)_k- M(n)_{k-1}.
\]  
Moreover, for any $j\le k$,
\begin{equation}\label{inverse}
 g^{-1}(M(n)_j)= M(k)_j\times \Hilb^{n-k}.
\end{equation}
It follows by induction that there are induced morphisms 
\[  g_i:  M(k)^{(i)}\times \Hilb^{n-k}   \to  M(n)_k^{(i)}, \quad 0\le i\le k. 
\]  

To show that $g_k$ is an isomorphism, we use induction on $k$. 

For $k=1$. Let $\sI_j$ be the ideal sheaf of $M(n)_j$ in $M(n)$. Recall that $M(n)_1^{(1)}$ is the strict transform of $M(n)_1$ under the blowing-up of $M(n)$ along $M(n)_0$. By the universal property of blowing up, for any scheme $T$ and any morphism $h: T\to M(n)_1$ such that $h^{-1}\sI_{0}\cdot\sO_T$ is invertible, there is a unique morphism $\widetilde h:T\to M(n)_1^{(1)}$ factoring $h$. To show that $g_1$ is an isomorphism, it suffices to show that there is a unique morphism 
\[ \widetilde t: T\to   M(1)^{(1)}\times \Hilb^{n-1}  
\] 
factoring $\widetilde h$. 
\[ \xymatrix{
 T\ar@/^2pc/^{\widetilde h}[rrr] \ar@{-->}@/_1pc/_{t_0}[drr]\ar[drrr]^(.7)h\ar@{-->}^-{\widetilde t}[rr]& &M(1)^{(1)}\times\Hilb^{n-1}  \ar[d] \ar[r]_-{g_1} &M(n)_1^{(1)} \ar[d]^{\pi_1}    \\
& & M(1)  \times \Hilb^{n-1} \ar[r]_-{g} & M(n)_1 \\
}\]

Note that by Lemma \ref{lemm_ideal}, $M(n)_1\subset S$ and $\sI_{0}=\sI_{Z_{m+1}}\cdot \sO_{M(n)_1}$. 
Recall that the ideal sheaf $\sI_{Z_{m+1}}$ is the image of the homomorphism 
\[ u_{m+1}: E_{m+1}\otimes \sO_{S}(-m-2)\to \sO_{S}
\] 
Its pullback to $T$ is $h^* (E_{m+1}\otimes \sO_{S}(-m-2))\to \sO_{T}$ whose image is $h^{-1}\sI_{0}\cdot \sO_T$.  Since $h^{-1}\sI_{0}\cdot \sO_T$ is invertible,  it induces a morphism $\tau: T\to \bP E_{m+1}$ over $S$. Observe that there are natural morphisms 
\[ M(1)\to \bP\Hom \Big(\mathop \bigwedge^{m+2}(V\otimes W_m), \mathop \bigwedge^{m+2}W_{1+m}\Big),
\]
\[ \Hilb^{n-1}\to  \bP\Hom \Big(\mathop \bigwedge^{m+2}W_{1+m}, \mathop \bigwedge^{m+2}W_{n+m}\Big).
\]
Since $\mathop \bigwedge^{m+2}W_{1+m}$ is a one dimensional vector space, taking composition of linear maps, we obtain a morphism 
\[ M(1)\times\Hilb^{n-1}\to \bP\Hom \Big(\mathop \bigwedge^{m+2}(V\otimes W_m), \mathop \bigwedge^{m+2}W_{n+m}\Big)=\bP E_{m+1}.
\]
Now since $\tau$ sends the open dense subscheme $T-h^{-1}(M(n)_{0})$ into the proper scheme $M(1)\times\Hilb^{n-1}$,  $\tau$ factors through $t_0$.  

Having $t_0$, note that $M(1)^{(1)} \times \Hilb^{n-1}$ is the blowing-up of $M(1) \times \Hilb^{n-1}$ along the subvariey $M(1)_0 \times \Hilb^{n-1}$ with ideal sheaf $\sJ_{0}=g^{-1}\sI_{0}\cdot \sO_{M(1) \times \Hilb^{n-1}}$, and by $gt_0=\pi_1\widetilde h$ and $\pi_1$ being a blowing-up, we see that $t_0^{-1}\sJ_{0}\cdot \sO_T$ is invertible. There is a unique morphism $\widetilde t:T\to M(1)^{(1)} \times \Hilb^{n-1}$ factoring $t_0$. 

For $k>1$, we need to use the result on the relations between ideal sheaves of $M(n)_j^{(i)}$ in $M(n)^{(i)}$. These relations are parallel to Theorem 2.4 (8) in \cite{V} and we omit it. 
 
\begin{rmk}
By Theorem \ref{thm_isom}, a point in $M(n)_k^{(k)}$ represents a sequence of nonzero homomorphisms  
\begin{align*}
&\psi_0:  V\otimes\sO_{\bP^1}\to \sO_{\bP^1}(n), \\  
&\psi_1:  \ker \psi_0  \to\coker\psi_0, \\
&\hskip 50pt \vdots  \\
&\psi_l:  \ker \psi_{l-1}  \to\coker\psi_{l-1},
\end{align*}
such that $l\le k$ and the lengths of $\coker \psi_0, \cdots, \coker \psi_l$ decrease to $n-k$.
\end{rmk}
 
\begin{thm}
For each $k=0,\cdots,n$, $M(n)_k^{(k)}$ is nonsingular. In particular, $M(n)^{(n)}$ is a nonsingular projective variety, together with $n$ smooth divisors 
\[ M(n)_0^{(n)}, M(n)_1^{(n)}, \cdots, M(n)_{n-1}^{(n)}.
\] 
\end{thm}

\begin{proof}
This result follows from Proposition \ref{prop_isom} and induction on $k$.  
\end{proof}

 
\section{Complete quotients on $\bP^1$}

Let $X=\bP^1$. Let $\sG$ be a coherent sheaf on $\bP^1$. Let $\sG^{\tor}$ be the torsion part of $\sG$ and $\sG^{\tf}=\sG/\sG^{\tor}$. Then $\sG^{\tf}$ is locally free. Recall that \cite{Hu} a complete quotient of $V\otimes \sO_{\bP^1}$ of rank $r$ and degree $n$ is either
\begin{enumerate}
\item  a quotient $V\otimes \sO_{\bP^1}\to \sG$ in $\Quot_{\bP^1}^{r,n}(V)$ such that $\sG$ is locally free;  or,
\item  a sequence 
\begin{align*} 
& \rho: V\otimes \sO_{\bP^1}\to \sG_0, \\  
& \zeta_1: 0\to \sG_0^\tf \to\sG_1\to \sG_0^\tor\to 0, \\  
&\hskip 50pt \vdots  \\
& \zeta_m: 0\to \sG_{m-1}^\tf \to \sG_{m}\to \sG_{m-1}^\tor\to 0
\end{align*} 
for some integer $m\geq 1$ such that $\rho$ is a quotient in $\Quot_{\bP^1}^{r,n}(V)$; $\zeta_i$ is a non-split extension for every $1 \le i \le m$; and the last sheaf $\sG_{m}$ is the unique one that is locally free.
\end{enumerate}
 
Let $B_{k}$ be a subscheme supported on the subset
\[ 
\{ [\rho: V\otimes \sO_{\bP^1} \to \sG] \in \Quot_{\bP^1}^{r,n}(V)\mid \text{the torsion part of
 ${\sG}$ has length $\ge k$} \}.
\]
Then 
\[ B_{n} \subset B_{n-1} \subset \cdots\subset B_1 \subset B_{0} = \Quot_{\bP^1}^{r,n}(V).
\]
It was shown in \cite{Hu} that the space of complete quotients on $\bP^1$ is the successive blowing-up of $\Quot_{\bP^1}^{r,n}(V)$ along the subschemes $B_n, B_{n-1}$ etc. 

It is well known that taking dual establishes a correspondence between stable pairs and quotients. This correspondence generalizes naturally to their complete version. 

\begin{prop}
There is a natural one-to-one correspondence between the set of equivalence classes of complete stable pairs $\{\varphi_i: V\otimes \sO_{\bP^1}\to \sF_i\}_{0\le i\le m}$ and the set of equivalence classes of complete quotients $\{ \rho: V\dual\otimes \sO_{\bP^1}\to \sG_0,\zeta_1,\cdots,\zeta_m\}$. 
\end{prop}

\begin{proof}
Let $(\sF_i,\varphi_i)_m$ be a complete stable pair. Construct a complete quotient $(\rho,\zeta_1,\cdots,\zeta_m)$ as follows.  Consider $\varphi_0:V\otimes \sO_{\bP^1}\to \sF_0$. Taking dual, we get $\varphi_0\dual:\sF_0\dual\to V\dual\otimes \sO_{\bP^1}$. Let $\sG_0$ be its cokernel. We get a quotient $\rho:V\dual\otimes \sO_{\bP^1}\to\sG_0$.  

Let $\sK_i=\ker\varphi_i$. Then a direct verification show that $\sG_0^{\tf}\cong \sK_0\dual$. There are inclusions of locally free sheaves
\[ \sK_m\subset \cdots\subset \sK_1\subset \sK_0\subset V\otimes \sO_{\bP^1}.
\]
Let $\sP_i=\coker(\sK_{i+1}\to\sK_i)$. It is a sheaf with finite support and 
\begin{equation}\label{pisom}
 \sP_i\cong \ker(\sR_{i+1}\to\sR_i)\cong\ker(\sT_i\to\sT_{i+1})
\end{equation}
where $\sR_i=\image \varphi_i$, $\sT_i=\coker\varphi_i$. Taking dual of $\sK_{i+1}\to\sK_i$, we have 
\begin{equation}\label{kext} 
0\to \sK_{i}\dual\to\sK_{i+1}\dual \to \sExt^1(\sP_i,\sO_X)\to 0.
\end{equation}
The second isomorphism in \eqref{pisom} induces a short exact sequence 
\[  0\to \sExt^1(\sT_{i+1},\sO_X)\to \sExt^1(\sT_i,\sO_X)\stackrel{\alpha}\lra \sExt^1(\sP_i,\sO_X)\to 0.
\]
Using $\alpha$ to pullback the exact sequence \eqref{kext}, we get a diagram 
\begin{equation}\label{cqdiag}
\xymatrix{
0\ar[r] 	&\sK_{i}\dual \ar[r] \ar@{=}[d] 	&\sG_{i+1}\ar[r]\ar[d] 	&\sExt^1(\sT_i,\sO_X) \ar[r]\ar[d]^{\alpha} 	&0 \\
0\ar[r] 	&\sK_{i}\dual \ar[r] 		&\sK_{i+1}\dual\ar[r] 	&\sExt^1(\sP_i,\sO_X) \ar[r] 			&0}
\end{equation}
Then by snake lemma, the middle column can be completed as
\[  0\to \sExt^1(\sT_{i+1},\sO_X)\to \sG_{i+1}\to \sK_{i+1}\dual\to 0.
\] 
It follows that $\sG_{i+1}^{\tor}\cong\sExt^1(\sT_{i+1},\sO_X)$ and $\sG_{i+1}^{\tf}\cong \sK_{i+1}\dual$. The top row of the diagram \eqref{cqdiag} becomes an extension $\zeta_{i+1}:　0\to \sG_{i}^{\tf}\to \sG_{i+1}\to \sG_{i}^{\tor}\to 0$. Because $\sT_m=0$, $\sG_m\cong \sK_m\dual$ and it is locally free.  This way we obtain a complete quotient $\{ \rho: V\dual\otimes \sO_{\bP^1}\to \sG_0,\zeta_1,\cdots,\zeta_m\}$. The inverse construction is similar and we omit it.
\end{proof}


\end{document}